\documentclass[amscd,amssymb,verbatim,12pt]{amsart}
\usepackage{amssymb,latexsym}
\usepackage{enumerate}
\usepackage{graphicx}
\usepackage[usenames,dvipsnames]{pstricks}
\usepackage{epsfig}
\usepackage{pst-grad}
\usepackage{pst-plot}

\newtheorem{definition}{Definition}
\newtheorem{theorem}{Theorem}
\newtheorem{proposition}{Proposition}
\newtheorem{lemma}{Lemma}
\newtheorem{remark}{Remark}
\newtheorem*{nonumbertheorem}{Theorem}
\newtheorem*{harpthm}{Theorem \ref{thm:harpe}}
\begin{document}

\title{On the end depth and ends of groups}

\author{M. Giannoudovardi}
\address{Department of Mathematics,
University of Athens, Athens, Greece}
\email{marthag@math.uoa.gr}

\begin{abstract}
We prove that any finitely generated one ended group has linear end depth. Moreover, we give alternative proofs to
theorems relating the growth of a finitely generated group to the number of its ends.
\end{abstract}

\maketitle

\section{Introduction}

The topology of a group at infinity is an important asymptotic invariant of groups (see \cite{GM}). In particular
the question which groups are simply connected at infinity is relevant to topology (\cite{Davis}, \cite{Wr}).
In order to study finitely generated groups that are simply connected at infinity (sci), Otera in \cite{otera},
introduces the function $V_1(r)$, measuring, ``in which way'' a group is sci.
The growth of this function, called sci growth and denoted by $V_1$,
is a quasi-isometry invariant for finitely generated groups. Expecting the existence of a group with super-linear $V_1$, Otera introduces the end depth function (see Definition \ref{defenddepth}), $V_0(r)$, a 0 - dimensional analogue of the sci growth for one ended groups.
 The end depth function measures 
 the ``depth'' of the bounded components of the complement of the ball $B(r)$ in the Cayley graph of the group.
 The growth of this function is a quasi-isometry invariant for finitely generated groups and it is called the
 end depth of the group. Otera \cite{otera} shows that given a group in which the growth of $V_0$
 is super-linear, we can construct a group where the function $V_1$ has super-linear growth:
\begin{nonumbertheorem}[Otera]
If $G=A\underset{H}{*}B$ is the free product with amalgamation over a (sub)group $H$ which is one ended with super-linear end depth and $A,B$ are sci groups, then $G$ is simply connected at infinity with super-linear $V_1$.
\end{nonumbertheorem}
One may also remark that a group with non-linear $V_0$ has
dead end elements (see \cite{clearytaback}). So a group with non-linear
$V_0$ has dead end elements with respect to any generating
set (to our knowledge there are no such examples in the
literature).

In this paper, we show that the function $V_0$ of any one ended group is linear (see Theorem
\ref{enddepth}).

In section \ref{ends} we give an alternative proof of the following theorem that was first proven by Erschler
\cite{harpe}, based on the Varopoulos inequality and answers
question VI.19, posed by Pierre de la Harpe in \cite{harpe}:
\begin{harpthm}
Let $G=\langle S\rangle$ be a finitely generated group and $X=\Gamma(G,S)$.
If there exists a sequence of positive integers $\{r_i\}_{i\geqslant 1}$ such that $\lim\limits_{i\to\infty}r_i=\infty$ and $\lim\limits_{i\to\infty}|S(r_i)|<\infty$, then $G$ is virtually cyclic.
\end{harpthm}
By $|S(r_i)|$ we denote the number of vertices in the sphere of
radius $r_i$ in the Cayley graph of the group $G$. Another proof, using different methods, was given by Timar \cite {Timar}. We give a proof, using elementary methods, without using the Varopoulos inequality.

In section \ref{lg} we show a stronger result. We relate the number
of ends of a finitely generated group with the growth of the spheres
 in its Cayley graph (see Theorem \ref{thm:sphere}). This Theorem is a weaker version of similar theorems proven by Justin in \cite{justin}, by Wilkie and Van Den Dries in \cite{wilandr} and by Imrich and Seifter in \cite{imrichseifter}. 
Also, in October 2009, Shalom and Tao proved a more general result for groups of polynomial growth in \cite{shalomtao}, by refining Kleiner's work in \cite{kleiner}.

\section{Preliminaries}
The definitions and notation introduced in this section will be used throughout this paper.\\
Let $(X,d)$ be a metric space and $A$, $B$ non-empty subsets of $X$. The \textit{distance of the sets} $A$ and $B$, $d(A,B)$ is defined by:
$$d(A,B)=\inf\{d(x,y)\mid x\in A,y\in B\}$$
We denote by $|A|$ the number of elements in  the set $A$. For  $r>0$, we define the $\textit{r}-$\textit{neighbourhood} of $A$, $N(A,r)$ by:
$$N(A,r)=\{y\in X\mid d(y,A)< r\}$$
For any $x\in X$, we denote by $S(x,r)$ ($B(x,r)$) the \textit{sphere} (\textit{ball}) in $X$ of radius $r$, centered at $x$.\\
We recall the definition of the number of ends of a metric space:

Let $(X,d)$ be a locally compact, connected metric space. For any compact subset, $K$, of $X$ we denote the number of unbounded connected components of
 $X\smallsetminus K$ by $e(X,K)$. Then, the \textit{number of ends} of $X$, denoted by $e(X)$, is the supremum, over all compact subsets $K$ of $X$, of $e(X,K)$:
$$e(X)=\sup\{e(X,K)\mid K\subset X compact\}$$
If $e(X)=1$, we say that $X$ is a \textit{one ended} metric space.

Let $G=\langle S\rangle$ be a finitely generated group. For any $A\subset G$ and $x\in G$ we denote by $xA$ the set $\{xa\mid a\in A\}$. Also, we denote by $\Gamma(G,S)$ the Cayley graph of $G$ with respect to the finite generating set $S$ and $d_S$ the word metric in $\Gamma(G,S)$. If $e$ is the identity element of $G$, for any positive integer $r$, we write $S(r)$ ($B(r)$) for the sphere $S(e,r)$ (ball $B(e,r)$) in $\Gamma(G,S)$. The \textit{size of a sphere} $S(g,r)$ in $X$ is the number of vertices (elements of $G$) in that sphere and we denote it by $|S(g,r)|$. We remark that $(\Gamma(G,S),d_S)$ is a locally compact, connected metric space, thus $e(\Gamma(G,S))$ is defined. It is a well known fact that this is independent of the finite generating set chosen. Thus, the \textit{number of ends}, $e(G)$, of $G$ is defined to be the number of ends of its Cayley graph, $\Gamma(G,S)$, with respect to a finite generating set $S$. 
We say that a finitely generated group is one ended if its Cayley graph, with respect to a finite generating set, is a one ended metric space. Note that the number of ends is a quasi-isometry invariant of finitely generated groups.\\
Regarding the number of ends of a finitely generated group, we recall the following important theorem of Hopf \cite{hhopf}:
\begin{theorem}\label{hopf}
A finitely generated group $G$ has either 0,1,2 or infinitely many ends.
\end{theorem}
It is clear that a finitely generated group $G$ is finite if and only if $e(G) = 0$.\\
On the other hand, from Stallings' classification Theorem \cite{stallings} we have that a finitely generated group $G$ has exactly two ends if and only if $G$ has an infinite cyclic, 
finite index subgroup. Therefore, we have the following equivalences:
$$e(G)=2\Leftrightarrow G \text{ is virtually }\mathbb{Z}\Leftrightarrow G\text{ is quasi isometric to } \mathbb{Z}$$
Finally we define the growth of a function, which we will need in section \ref{edf}:\\
Let $f,g:\mathbb{R}_+\to\mathbb{R}_+$. We say that the \textit{growth of the function} $f$ is at most the growth of the function $g$ and
 we write $f\prec g$, if there exist real constants $a_1>0,a_2>0,a_3$ such that, for any $x\in\mathbb{R}_+$, the following inequality holds:
$$f(x)\leqslant a_1g(a_2x)+a_3$$ 
The functions $f$ and $g$ have the \textit{same growth}, denoted by $f\sim g$, if $f\prec g$ and $g\prec f$.\\
Note that the relation $f\sim g$ is an equivalence relation. The \textit{growth rate} of a function $f$ is defined as the corresponding equivalence class of the function $f$.
Lastly, we say that $f$ has \textit{linear growth} if $f(x)\sim x$.

\section{The End Depth Function}\label{edf}
In this section we examine the growth of the end depth function of a one ended group. We remark that this notion is a $0$-dimensional analogue of the sci growth for one ended groups and it was introduced by Otera \cite{otera}.\\
We start by giving the definition of the end depth function that is due to Otera.
\begin{definition}\label{defenddepth}Let $G=\langle S\rangle$ be a finitely generated one ended group and $X=\Gamma(G,S)$. For any $r> 0$, we denote by $N(r)$ the set of all $k\in\mathbb{R}$ such that any two points in $X\smallsetminus B(k)$ can be joined by a path
 outside $B(r)$. The function $V^X_0(r)=\inf N(r)$ is called the \textit{end depth function} of $X$.
\end{definition}
The idea of the end depth function can be grasped more easily if we consider the bounded connected components of $X\smallsetminus B(r)$:
\begin{remark}\label{remark2} Let $G=\langle S\rangle$ be a finitely generated group, $X=\Gamma(G,S)$, $d=d_S$ and $e$ the identity element of $G$. 
For any $r>0$ the set $X\smallsetminus B(r)$ has finitely many connected components. We denote by $U_r$ the unique unbounded connected component 
and by $B_r$ the union of the bounded components of $X\smallsetminus B(r)$.
\begin{figure}[ht]
    \centering
        \includegraphics[scale=0.8]{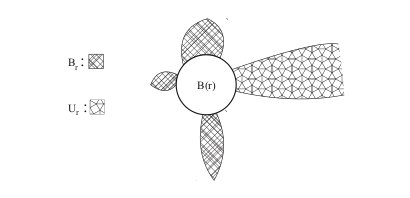}
       \caption{}
\label{Fig:1}
\end{figure}\\
Then, we have the following:
\begin{enumerate}
\item Clearly, $B_r=\emptyset$ if and only if $V_0^X(r)=r$.
\item Suppose that $B_r\neq\emptyset$, i.e. $X\smallsetminus B(r)$ has at least one bounded connected component. 
Then, for any $x\in B_r$ any path that joins $x$ to an element $y\in U_r$ must pass through $B(r)$. 
Thus, for any $x\in B_r$, $V_0^X(r)\geqslant d(e,x)$, so:
$$V^X_0(r) \geqslant\max\{d(e,x)\mid x\in B_r\}$$
On the other hand, for any $y,z\in X$ with $d(e,y),d(e,z)>\max\{d(e,x)\mid x\in B_r\}$ we have that $y,z\in U_r$. This implies that $y$ and $z$ can be joined by a path outside $B(r)$, so $y,z\in X\smallsetminus B(V^X_0(r))$.
It follows that:
$$V^X_0(r) =\max\{d(e,x)\mid x\in B_r\}$$
From the latter equality we see how, in a sense, $V^X_0$ measures the depth of the bounded connected components of $X\smallsetminus B(r)$.\\
Furthermore, there exists a bounded connected component, $A_r$, of $X\smallsetminus B(r)$ and an element $a\in A_r$ such that
$$V^X_0(r)=d(e,a)=\max\{d(e,x)\mid x\in B_r\}$$
\begin{figure}[ht]
    \centering
        \includegraphics[scale=0.8]{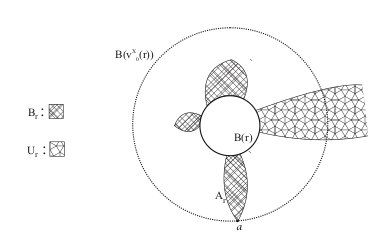}
       \caption{}
\label{Fig:2}
\end{figure}
\end{enumerate}
\end{remark}
The end depth function depends on the choice of the generating set, but its growth rate does not. Actually, it is a quasi-isometry invariant for finitely generated groups \cite{otera}. 
Therefore, we recall the following definition that is due to Otera \cite{otera}.
\begin{definition}
Let $G=\langle S\rangle$ be a one ended group and $X=\Gamma(G,S)$.
The \textit{end depth} of $G$ is the growth rate of the function $V_0^X$.
\end{definition}

\begin{theorem}\label{enddepth}
The end depth of a one ended group is linear.
\end{theorem}
\begin{proof} 
Let $G=\langle S\rangle$ be a one ended group, $X=\Gamma(G,S)$ and $d=d_S$. We argue by contradiction that, for any integer $r\geqslant 2$, $V_0^X(r)\leqslant 4r$.\\
Suppose that there is a positive integer $r\geqslant 2$, such that $V_0^X(r)>4r$. Then, as stated in Remark \ref{remark2}, there exists a bounded connected component, $A$, of $X\smallsetminus B(r)$ and an element $a\in A$, such that $V_0^X(r)=d(e,a)$. 
Note that $d(a,B(r))>3r$, therefore $d(B(a,r),B(r))>2r$.\\
We consider the left action of $a$ on $X$. Then $aB(r)=B(a,r)\subset A$ and $aA\cap A\neq\emptyset$. Moreover, since $aA\cap aB(r)=\emptyset$ and $|A|=|aA|$, 
we have that $aA\smallsetminus A\neq\emptyset$. Therefore $aA\cap B(r)\neq\emptyset$. \\
Recall that $G$ is one ended, so there exists a unique unbounded connected component, $U$, of $X\smallsetminus B(r)$ and an infinite geodesic path 
$\gamma = (\gamma_0,\gamma_1,\dots)$ of vertices in $U$, such that $d(\gamma_0,B(r))=1$.
\begin{figure}[ht]
    \centering
        \includegraphics[scale=0.7]{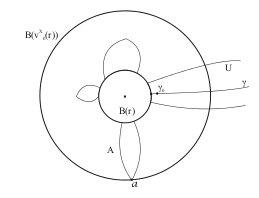}
       \caption{}
\label{Fig:3}
\end{figure}
Clearly, $d(a\gamma_0,a)=r+1$ and since $d(a,X\smallsetminus A)>3r$, it follows that $a\gamma_0\in A$. 
On the other hand, since $a\gamma=(a\gamma_0,a\gamma_1,\dots)$ is an infinite path while $B(r)\cup B_r$, where $B_r$ is the union of the connected components of $X\smallsetminus B(r)$, is finite, there exists $n>0$, such that 
$a\gamma_{n}\in U$. 
\begin{figure}[ht]
    \centering
        \includegraphics[scale=0.7]{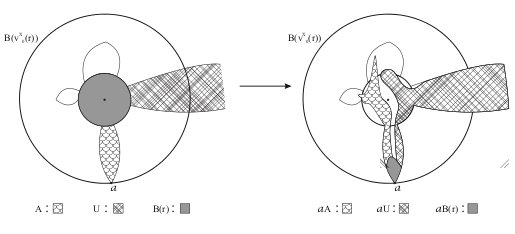}
       \caption{}
\label{Fig:4}
\end{figure}
Therefore, the path $\gamma'=(a\gamma_0,a\gamma_1,\dots,a\gamma_{n})$ joins an element of $A$ to an element of $U$. 
But $A$ and $U$ are connected components of $X\smallsetminus B(r)$, 
so $\gamma'$ passes through $B(r)$. Thus, there exists $m\in\{0,1,\dots,n\}$, such that $y=a\gamma_{m}\in B(r)$.\\
Let $x\in aA\cap B(r)$. Then $x=az$, for some $z\in A$. The elements $x$ and $y$ are joined by a path 
$\varepsilon = (\varepsilon_0=x,\varepsilon_1,\dots,\varepsilon_k=y)$ in $B(r)$, for some $k\in\mathbb{N}$. The sequence:
$$\varepsilon'=a^{-1}\varepsilon=(a^{-1}\varepsilon_0,a^{-1}\varepsilon_1,\dots,a^{-1}\varepsilon_k)$$
is a path that joins $a^{-1}x=z\in A$ to $a^{-1}y=\gamma_{m}\in U$. Therefore, $\varepsilon'$ passes through $B(r)$. Thus, there exists $j\in\{1,2,\dots,k\}$ 
such that $a^{-1}\varepsilon_{j}\in B(r)$. But then, $\varepsilon_{j}\in B(r)\cap aB(r)\subset B(r)\cap A$, which is a contradiction since $B(r)\cap A=\emptyset$.\\
In conclusion, for any integer $r\geqslant 2$, we have that $V_0^X(r)\leqslant 4r$. Hence, $V_0^X$ has linear growth.
\end{proof}

\section{On Ends of Groups}\label{ends}
The main objective of this section is to present an alternative approach to question VI.19, posed by Pierre de la Harpe in \cite{harpe}. 
Theorem \ref{thm:harpe} answers this 
question and it was first proven by Erschler
\cite{harpe}, based on the Varopoulos inequality, and later by Timar \cite {Timar}, using different methods.
We give a geometric proof, without using the Varopoulos inequality. 
\begin{proposition}\label{OBSS}
Let $G=\langle S\rangle$ be a finitely generated group and $X=\Gamma(G,S)$. Suppose that there is a positive integer $n$ and a sequence of positive integers $\{r_i\}_i$ so that, 
for any $i\in\mathbb{N}$, there exists a compact
subset $K_i$ of $X$, with the following properties:
\begin{enumerate}
\item $diam(K_i)<n$
\item $N(K_i,r_i)\smallsetminus K_i$ has at least two connected components, $A_i$ and $B_i$
\item $\lim\limits_{i\to\infty}r_i=\lim\limits_{i\to\infty}diam(A_i)=\lim\limits_{i\to\infty}diam(B_i)=\infty$
\end{enumerate}
Then $e(G)>1$.
\end{proposition}

\begin{figure}[ht]
   \centering
       \includegraphics[scale=1]{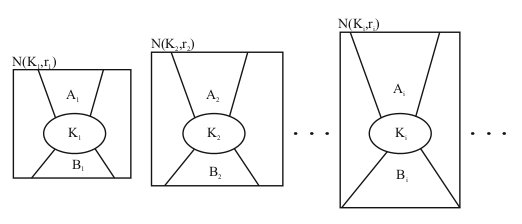}
       \caption{}
    \label{fig:5}
\end{figure}

\begin{proof}
We may assume that, for all $i\geqslant 1$, $K_i$ is a graph.\\
$G$ is finitely generated and for any i, the set $K_i$ has diameter less than $n$, so the number of edges in $K_i$ is less than $|S|^{2n}$. Therefore, there exists a 
subsequence, $\{K_{i_j}\}_j$, such that any two sets of this subsequence are isometric. 
We re-index this subsequence to avoid double indices, so we write $K_j$ for $K_{i_j}$. The action of $G$ on $X$ is by isometries, so there exists a subsequence of $\{K_{j}\}_j$, that we still denote by $\{K_j\}_j$ for convenience, so that for any $j>1$ there is $g_j\in G$ such that $g_jK_{j}=K_{1}$. 
Again, as $G$ is finitely generated and for any $j>1$, $diam(K_j)< n$, we conclude that the number of connected components of 
$X\smallsetminus K_j$ is uniformly bounded. Therefore, there exists yet another subsequence of $\{K_{j}\}_j$, denoted also for convenience by $\{K_{j}\}_j$, 
so that
$$\bigcap\limits_{j>1} g_jA_{j}\cap A_{1}\neq\emptyset\quad\text{and}\quad\bigcap\limits_{j>1} g_jB_{j}\cap B_{1}\neq\emptyset$$
Now, let $A$ and $B$ be connected components of $X\smallsetminus K_{1}$ such that $A_{1}\subset A$ and $B_{1}\subset B$. 
Then, for all $j$, we have that $g_jA_{j}\subset A$ and $g_jB_{j}\subset B$, so $diam(A)\geqslant diam(A_{j})$ and $diam(B)\geqslant diam(B_{j})$. This implies that
$diam(A)=\infty$ and $diam(B)=\infty$.

Finally, we will argue by contradiction that $A$ and $B$ are different connected components of $X\smallsetminus K_{1}$. 
So, suppose that $A$ and $B$ are connected in $X\smallsetminus K_{1}$. Let $x\in\bigcap\limits_{j>1} g_jA_{j}\cap A_{1}$ and 
$y\in\bigcap\limits_{j>1} g_jB_{j}\cap B_{1}$, so that $d_S(x,K_{1})=d_S(y,K_{1})=1$. Then, there exists a finite path, $\gamma$, of length $l\in\mathbb{N}$ in $X\smallsetminus K_{1}$ that joins 
$x$ to $y$. Clearly, for any $j>1$, $\gamma_j=g_j^{-1}\gamma$ is a finite path of length $l$, that joins $x_j=g_j^{-1}x\in A_j$ to $y_j=g_j^{-1}y\in B_j$. Thus there exists $m\in\mathbb{N}$ so that, for any $j>m$, the path $\gamma_j$ is contained in $N(K_{j},r_j)$. Note that, for any $j>m$, the elements $x_j$ and $y_j$ are connected outside $N(K_j,r_j)$ and that their distance from $X\smallsetminus N(K_j,r_j)$ is greater than $r_j-2$. Therefore, we reach to the conclusion that, for any $j>m$, $l>r_j-2$. 
This however contradicts our hypothesis that $\lim\limits_{i\to\infty}r_i=\infty$.

Hence, $K_{1}$ is a compact subset of $X$ with, at least, two unbounded connected components. Thus $e(G)>1$.
\end{proof}

\begin{remark}\label{remark1}
It is worth mentioning that Proposition \ref{OBSS}, does not hold for arbitrary metric spaces. 
For example, consider the space $X=[0,\infty)$ with the usual metric. Then, $X$ is a one ended metric space and it is easy to see that all the conditions of Proposition \ref{OBSS} hold for $X$:\\
For any $r\in\mathbb{N}$, we set $K_r=[ r+1,r+2]$.
Then, for any $r\in\mathbb{N}$, $K_r$ is a compact subset of $X$ with $diam(K_r)<2$. Moreover, the connected components of $X\smallsetminus K_r$ are the sets  
$A_r = [0,r+1)$ and $B_r=(r+2,\infty)$, with $diam(A_r)> r$ and $diam(B_r)=\infty$.
\begin{figure}[ht]
   \centering
       \includegraphics[scale=1]{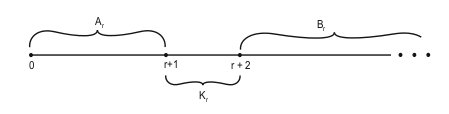}
       \caption{}
    \label{Fig:6}
\end{figure}
\end{remark}

In the following theorem we will use Proposition \ref{OBSS} to give, as mentioned in the introduction,
 an alternative approach to question VI.19 posed by Pierre de la Harpe in \cite{harpe}.

\begin{theorem}\label{thm:harpe}
Let $G=\langle S\rangle$ be a finitely generated group and $X=\Gamma(G,S)$.
If there exists a sequence of positive integers $\{r_i\}_{i\geqslant 1}$ such that $\lim\limits_{i\to\infty}r_i=\infty$ and $\lim\limits_{i\to\infty}|S(r_i)|<\infty$, then $G$ is virtually cyclic.
\end{theorem}
In the case that $G$ is infinite, we will, upon passing to a subsequence, split, for any $t$, the set $S(r_t)$ into 2 subsets, $K_t$ and $F_t$, whose distance tends to infinity and so that $\{diam(K_t)\}_t$ is bounded. Finally, we show that we can apply Proposition \ref{OBSS} for the sets $K_t$.
\begin{proof}
This is trivial if $G$ is finite, so suppose that $G$ is infinite, thus $e(G)\geqslant 1$.\\
For simplicity, let $d=d_S$. There exists a bi-infinite geodesic path, $\gamma = (\dots,\gamma_{-1},\gamma_0, \gamma_1,\dots)$, of vertices in $X$, 
where $\gamma_0$ is the identity element of $G$. For all $i\geqslant 1$, $X\smallsetminus S(r_i)$ has an unbounded connected component $U_i$, 
such that, for any $j\geqslant r_i+1$
$$\gamma_{j}\in U_i$$
Obviously, then, for all $i\geqslant 1$
$$U_{i+1}\subset U_i$$
On the other hand, $\lim\limits_{i\to\infty}|S(r_i)|<\infty$, so the sequence $\{|S(r_i)|\}_i$ is bounded. Hence $e(G)<\infty$ and there exist a positive integer $m$ and a subsequence $\{r_{i_l}\}_l$, such that $\lim\limits_{l\to\infty}r_{i_l}=\infty$ and, for all $l>0$, $|S(r_{i_l})|=m$. As usual, we re-index this subsequence to avoid double indices, so we write $r_l$ for $r_{i_l}$. 
For any $l>0$, let 
$$S(r_l)=\{x_1(l),x_2(l),\dots,x_m(l)\}$$
We may assume that, for all $l>0$, $x_1(l)=\gamma_{r_l}$ and $x_m(l)=\gamma_{-r_l}$. For any $l>0$, we consider, for all $j,k\in\{1,\dots,m\}$, the distances $d(x_j(l),x_k(l))$.
Using a diagonal argument, we can extract a subsequence $\{r_{l_t}\}_t$, so that the distances of any two elements of the sphere $S(r_{l_t})$ converge in $\mathbb{R}\cup\{\infty\}$. Again, to keep notation simple we denote this sequence by $\{r_{t}\}_t$. Therefore, we have that, for any $j,k\in\{1,\dots,m\}$:
$$\lim\limits_{t\to\infty}d(x_j(t),x_k(t)) = a_{jk}\in\mathbb{R}\cup\{\infty\}$$
This implies that there exists a partition $\mathcal{P}$ of the set $\{1,\dots,m\}$ so that any $j,k\in\{1,\dots,m\}$ belong to the same set of the partition if and only if $a_{jk}<\infty$.\\
We note that:
$$a_{1m}=\lim\limits_{t\to\infty}d(\gamma_{r_t},\gamma_{-r_t})=\infty$$
Therefore, the partition $\mathcal{P}$ is not a singleton.\\
Now, let $Y\in\mathcal{P}$. For any positive integer $t$, we define the $t-$corresponding set of $Y$ in $X$ as follows:
$$Y_t=\{x_j(t)\mid j\in Y\}$$
Let $K\in\mathcal{P}$, so that $1\in K$. We will show that we can apply Proposition \ref{OBSS} for the $t-$corresponding sets $K_t$. For any $t>0$, $x_1(t)=\gamma_{r_t}\in K_t$, thus $d(K_t,U_t)=1$. Furthermore:
$$\lim\limits_{t\to\infty}diam(K_t)=\lim\limits_{t\to\infty}\sup \{d(x_j(t),x_k(t))\mid j,k\in K\}<\infty$$
So, there exists $M>0$ such that, for all $t>0$,
$$diam(K_t)<M$$
We denote the set $\{1,\dots,m\}\smallsetminus K$ by $F$. Considering that $\mathcal{P}$ is not a singleton, we have that $F\neq\emptyset$. The $t-$corresponding set of $F$ 
in $X$ is the set $F_t=S(r_t)\smallsetminus K_t$. For any $t>0$, we set
$$D_t=\frac{1}{2}d(F_t,K_t)$$
Note that, since $K$ and $F$ are distinct, non empty sets of $\mathcal{P}$, 
we have that:
$$\lim\limits_{t\to\infty}D_t=\frac{1}{2}\lim\limits_{t\to\infty}\inf\{d(x_j(t),x_k(t))\mid j\in F, k\in K\}=\infty$$
Without loss of generality we assume that, for all $t>0$, $D_t>2$.
For any $t>0$, we let
$$N_t=N(K_t,D_t)$$
Then, for any $t>0$, we have that $d(N_t,F_t)>1$, thus $N_t\cap F_t=\emptyset$.\\
Finally, for any $t>0$, let
$$A_t=N_t\cap U_t\quad \text{and}\quad B_t=N_t\cap(B(r_t)\smallsetminus K_t)$$ 
Then, for all $t>0$, $\gamma_{r_t+1}\in A_t$ and $\gamma_{r_t-1}\in B_t$, so the sets $A_t$ and $B_t$ are non-empty. Moreover, it is immediate from the definitions that the sets $A_t$ and $B_t$ are different connected components of $N_t\smallsetminus K_t$.
\begin{figure}[ht]
   \centering
       \includegraphics[scale=1]{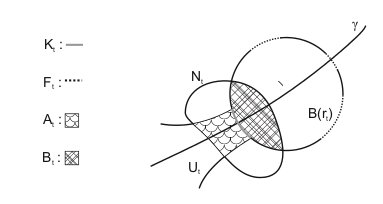}
       \caption{}
    \label{fig:7}
\end{figure}\\
Finally, for any $t>0$, we have that $\gamma_{r_t+D_t-1}\in A_t$ and $\gamma_{r_t-D_t+1}\in B_t$. Hence:
$$diam(A_t)\geqslant d(\gamma_{r_t+1},\gamma_{r_t+D_t-1})=D_t-2$$
and 
$$diam(B_t)\geqslant d(\gamma_{r_t-1},\gamma_{r_t-D_t+1})=D_t-2$$
therefore
$$\lim\limits_{t\to\infty}diam(A_t)=\lim\limits_{t\to\infty}diam(B_t)= \infty$$
From Proposition \ref{OBSS}, it follows that $e(G)>1$. We recall that $e(G)<\infty$, so from Hopf's Theorem (Theorem \ref{hopf}) we derive that $e(G)=2$, thus $G$ is virtually $\mathbb{Z}$.
\end{proof}

\section{Linear Growth}\label{lg}
The main objective of this section is to give a characterization for groups that are virtually cyclic. More specifically, in Theorem \ref{thm:sphere}, we give a condition for the growth of spheres in $G$ 
that results to $G$ being virtually $\mathbb{Z}$. This theorem is, as mentioned in the introduction, a weaker version of theorems proven by Justin \cite{justin},
 Wilkie and Van Den Dries \cite{wilandr}, Imrich and Seifter \cite{imrichseifter}, Shalom and Tao \cite{shalomtao}. The techniques used in this paper are quite elementary 
and the proof we give has a strong geometric flavour.\\
We start by giving some definitions.

\begin{definition}
Let $(Y,d)$ be a metric space, $a,m$ positive integers, with $m\geqslant 2$, and $A_1,\dots,A_m$ non empty subsets of $Y$. 
We say that $\mathcal{A}=(\{A_i\}_i,m,a)$ is a \textit{gl-partition} of the space $Y$, if
\begin{enumerate}
\item for any $i,j\in\{1,\dots,m\}$, either $A_i=A_j$ or $A_i\cap A_j=\emptyset$
\item $Y=\bigsqcup\limits_{i=1}^{m}A_i$
\item for any $i\in\{1,\dots,m\}$, $d(A_i,Y\smallsetminus A_i)>a\cdot\max\{diam(A_j),1\mid j=1,\dots,m\}$
\end{enumerate}
\end{definition}

\begin{definition}
Let $(Y,d)$, $(Z,d')$ be metric spaces with gl-partitions $\mathcal{A}=(\{A_i\}_i,k_1,a)$, $\mathcal{B}=(\{B_i\}_i,k_2,b)$ respectively. 
We say that $\mathcal{A}$ and $\mathcal{B}$ are \textit{similar gl-partitions}, if:
\begin{enumerate}
\item $a=b$
\item $k_1=k_2$
\item After some rearrangement if necessary, for all $i=1,\dots,k_1$, $A_i$ is isometric to $B_i$
\end{enumerate}
\end{definition}

\begin{remark}
It is an immediate consequence of the definitions that if $(Y,d)$ and $(Z,d')$ are isometric metric spaces and $\mathcal{A}$ is a gl-partition of $Y$, then there 
exists a gl-partition of $Z$, similar to $\mathcal{A}$.
\end{remark}

We state now a lemma that gives an insight to the structure of a finite metric space that has big diameter compared to the number of its elements.

\begin{lemma}[Distant Galaxies Lemma]\label{dgl}
Let $(Y,d)$ be a finite metric space and $a\in\mathbb{Z}$ greater than 2. Suppose that $Y$ has $n$ elements and diameter greater than $(2a+1)^{n+2}$. Then there exists a 
gl-partition $\mathcal{A}=(\{A_i\}_i,n,a)$ of $Y$.
\end{lemma}
What we state in this lemma is intuitively obvious, since one expects that if we have very few points to distribute on a great distance, then distant groups of points will be 
formed, forming in a way distant galaxies in $Y$.
\begin{proof}
Suppose that $Y=\{y_1,\dots,y_n\}$. For any $i=1,\dots,n$, we set
$$A_0(i) = \{y_i\}\quad\text{and}\quad d_0=1$$
and we define inductively, for any positive integer $m$:
\begin{center}\begin{tabular}{l l}
& $A_m(i) = \{y\in Y\mid d(y,A_{m-1}(i))\leqslant a\cdot d_{m-1}\}$\\
$\phantom{x}$ & \\
& $d_m = \max\limits_{1\leqslant j\leqslant n} \{diam(A_m(j)),1\}$
\end{tabular}\end{center}
Then, for any $m>0$ and $i=1,\dots,n$, we have that
$$d_m \leqslant diam(A_{m-1}(i))+2ad_{m-1}$$
thus,
$$d_m\leqslant (2a+1)d_{m-1}$$
and finally
$$d_m\leqslant (2a+1)^m$$
Since $Y$ has $n$ elements, for any $i=1,\dots,n$, the sequence $\{A_m(i)\}_m$ is finally constant. So, let $k$ be the minimum positive integer such that, for all $i=1\dots,n$, we have that $A_k(i)=A_{k+1}(i)$. 
We then denote the set $A_k(i)$ by $A_i$.\\
Note that, from the construction of the sets $\{A_i\}_i$, we have that, for any $i=1,\dots,n$,
$$d(A_i,Y\smallsetminus A_i)> a \cdot\max\{diam(A_i),1\mid i=1,\dots,n\}$$
We will show that, for any $i\neq j\in\{1\dots,n\}$, either $A_i\cap A_j=\emptyset$ or $A_i=A_j$.\\
Let $i\neq j\in\{1,\dots,n\}$ such that $A_i\cap A_j\neq\emptyset$. Then, for any $y\in A_i$ we have that $d(y,A_j)\leqslant d_k$, so $y\in A_{k+1}(j)=A_j$.
 Therefore $A_i\subset A_j$. Similarly, get that $A_j\subset A_i$, hence $A_i=A_j$.\\
In order to proceed we have to show that the steps needed to define the sets $A_i$ are at most $n+1$. 
In each step prior to the $k^{th}$, at least two of the sets $\{A_m(i)\}_{i=1,\dots,n}$ have a non-empty intersection. So, eventually these two sets get identified. An example is illustrated in figure \ref{Fig:8}.
\begin{figure}[ht]
   \centering
       \includegraphics[scale=0.8]{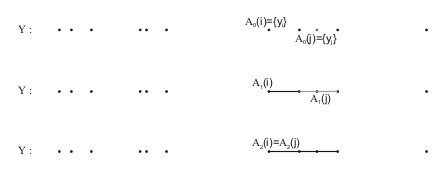}
       \caption{}
    \label{Fig:8}
\end{figure}\\
Therefore, we need at most $n+1$ steps to define the sets $A_i$, so $k\leqslant n+1$.\\
Finally, we will show that, for any $i=1,\dots,n$, $A_i\neq Y$:\\
Suppose, to the contrary, that there exists $i\in\{1,\dots,n\}$, such that $A_i=Y$. Then:
$$diam(Y)= diam (A_i)\leqslant d_k \leqslant (2a+1)^{k}\Rightarrow$$
$$(2a+1)^{n+2}\leqslant (2a+1)^{k}\Rightarrow k\geqslant n+2$$
But this contradicts the fact that $k$ is at most $n+1$.

Therefore, $A_i\neq Y$, for any $i\in\{1,\dots,n\}$. We conclude that $\mathcal{A}=(\{A_{i_j}\},n,a)$ is a gl-partition of the metric space $Y$.
\end{proof}

\begin{theorem}\label{thm:sphere}
Let $G=\langle S\rangle$ be a finitely generated group and $X=\Gamma(G,S)$. If there are 
$a,n\in\mathbb{N}$ with $a\geqslant 100$ and $n\geqslant 2$, such that a sphere of radius $(2a+1)^{n+2}$ in $X$ has at most $n$ elements, then $G$ is virtually cyclic.
\end{theorem}

\begin{proof}
This is trivial if $G$ is finite, so suppose that $G$ is infinite, thus $e(G)\geqslant 1$.\\
For simplicity, let $d=d_S$. There exists a bi-infinite geodesic path, $\gamma = (\dots,\gamma_{-1},\gamma_0, \gamma_1,\dots)$, of vertices in $X$, 
where $\gamma_0$ is the identity element of $G$. Let $r=(2a+1)^{n+2}$ and for any $i\in\mathbb{Z}$ denote the sphere 
$S(\gamma_i,r)$ by $S_i$.\\
Since $\gamma_{i+r},\gamma_{i-r}\in S_i$, we get that $diam(S_i)>r$. Therefore, from the Distant Galaxies Lemma, 
it follows that for any $i\in\mathbb{Z}$, there exists a gl-partition, of the set $S_i$. On the other hand, for any $i,j\in\mathbb{Z}$, the sets $S_i$ and $S_j$ are isometric, 
so there exist similar gl-partitions of these sets. Thus, for any $i\in\mathbb{Z}$, let $\mathcal{A}_i=(\{A_l(i)\}_l,k,a)$ be a gl-partition of $S_i$, such that for any 
$j\in\mathbb{Z}$, $\mathcal{A}_i$ and $\mathcal{A}_j$ are similar. Let
$$D_i=\max\{diam(A_l(i)),1\mid l=1,\dots,k\}$$
For any $i,j\in\mathbb{Z}$, since $\mathcal{A}_i$ and $\mathcal{A}_j$ are similar gl-partitions, $D_i=D_j$, so we denote $D_i$ by $D$. Also, for any $i\in\mathbb{Z}$, we denote by $A_i$ the set of the gl-partition $\mathcal{A}_i$, that $\gamma_{i+r}$ belongs to.
\begin{figure}[ht]
   \centering
       \includegraphics[scale=0.7]{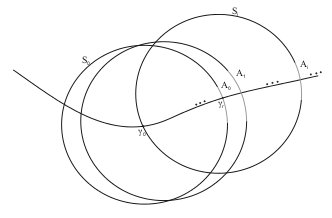}
       \caption{}
    \label{Fig:9}
\end{figure}\\
Then, for any $i\in\mathbb{Z}$, we have that
$$diam(A_i)\leqslant D\quad\text{and}\quad d(A_i, S_i\smallsetminus A_i)>aD$$
We note that $r>aD\geqslant 100D$.\\

Let $B=B(\gamma_0,39D)$ and $x\in B$, then:
\begin{enumerate}[(a)]
\item $|d(x,\gamma_i)-d(x,\gamma_{i+1})|\leqslant d(\gamma_i,\gamma_{i+1})=1,\quad$for any $i\in\mathbb{Z}$
\item $d(x,\gamma_{40D-r})\leqslant d(x,\gamma_0)+d(\gamma_0,\gamma_{40D-r})\leqslant 39D+r-40D<r$
\item $d(x,\gamma_{-40D-r})\geqslant d(\gamma_0,\gamma_{-40D-r})-d(\gamma_0,x)\geqslant 40D+r-39D>r$
\end{enumerate}
Therefore, there exists $m\in\{-40D-r,\dots,40D-r\}$, so that $d(x,\gamma_m)=r$, thus $x\in S_m$.\\
Furthermore:
$$d(x,A_m)\leqslant d(x,\gamma_{r+m})\leqslant d(x,\gamma_0)+d(\gamma_0,\gamma_{r+m})\leqslant 79D<aD$$
Thus, $x\in A_m$ and $d(x,\gamma_{r+m})\leqslant D$, where $r+m\in\{-40D,\dots,40D\}$. We have shown therefore that
$$B\subset \bigcup\limits_{i=-40D}^{40D} S(\gamma_i,D)$$
\begin{figure}[ht]
   \centering
       \includegraphics[scale=1]{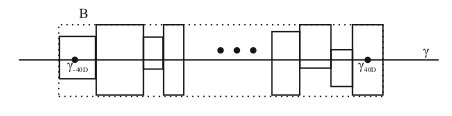}
       \caption{}
    \label{Fig:10}
\end{figure}\\
Now, since balls of the same radius in $X$ are isometric, by moving along the path $\gamma$ we can easily see that
$$X=\bigcup\limits_{i\in\mathbb{Z}} S(\gamma_i,D)$$
Therefore $X$ is quasi-isometric to $\mathbb{Z}$, so $e(G)=2$.
\end{proof}

\subsection*{ACKNOWLEDGMENTS} The author is grateful to Panagiotis Papazoglou for valuable conversations, for his encouragement and for his essential remarks on an earlier version of this paper.\\
This is a part of the author's PhD work and it was partially supported by the Greek Scholarship Foundation.

\end{document}